\documentclass[10pt]{article}
\usepackage{graphics,psfrag}
\usepackage{latexsym}
\usepackage{graphicx}
\usepackage[sf, bf, small]{titlesec}
\usepackage{relsize}
\catcode`@=11
\renewcommand{\@begintheorem}[2]{\it \trivlist
      \item[\hskip \labelsep{\bf #1\ #2{\rm :}}]}
\renewcommand{\@opargbegintheorem}[3]{\it \trivlist
      \item[\hskip \labelsep{\bf #1\ #2\ {\rm (#3)\/:}}]}
\def\@sect#1#2#3#4#5#6[#7]#8{\ifnum #2>\c@secnumdepth
     \def\@svsec{}\else
     \refstepcounter{#1}\edef\@svsec{\csname the#1\endcsname{.}\hskip 1em }\fi
     \@tempskipa #5\relax
      \ifdim \@tempskipa>\z@
        \begingroup #6\relax
          \@hangfrom{\hskip #3\relax\@svsec}{\interlinepenalty \@M #8\par}
        \endgroup
       \csname #1mark\endcsname{#7}\addcontentsline
         {toc}{#1}{\ifnum #2>\c@secnumdepth \else
                      \protect\numberline{\csname the#1\endcsname}\fi
                    #7}\else
        \def\@svsechd{#6\hskip #3\@svsec #8\csname #1mark\endcsname
                      {#7}\addcontentsline
                           {toc}{#1}{\ifnum #2>\c@secnumdepth \else
                             \protect\numberline{\csname the#1\endcsname}\fi
                       #7}}\fi
     \@xsect{#5}}

\@addtoreset{equation}{section}
\newcommand{\Delete}[1]{}

\usepackage{amsmath,amssymb,amsthm}

\theoremstyle{plain}
\newtheorem{Thm}{Theorem}
\newtheorem{Lem}[Thm]{Lemma}

\usepackage{graphpap}
\usepackage{pstricks}
\usepackage{pst-node}

\begin{document}

\title{The niche graphs of doubly partial orders}

\author{\begin{tabular}{c}
{\sc Suh-Ryung KIM}
\thanks{This work was supported by the Korea Research Foundation Grant
funded by the Korean Government (MOEHRD) (KRF-2008-531-C00004).}
\qquad {\sc Jung Yeun LEE}
\addtocounter{footnote}{-1}\footnotemark%
\qquad {\sc Boram PARK}
\thanks{The author was supported by Seoul Fellowship.}\ \thanks{corresponding author: kawa22@snu.ac.kr ; borampark22@gmail.com} \\
[1ex]
\small{Department of Mathematics Education,
Seoul National University, Seoul 151-742, Korea.} \\
\\
{\sc Won Jin PARK}\\
[1ex]
\small{Department of Mathematics,
Seoul National University, Seoul 151-742, Korea.} \\
\\
{\sc Yoshio SANO}
\thanks{The author was supported by JSPS Research Fellowships
for Young Scientists.
The author was also supported partly by Global COE program
``Fostering Top Leaders in Mathematics".}\\
[1ex]
\small{Research Institute for Mathematical Sciences,
Kyoto University, Kyoto 606-8502, Japan.}
\end{tabular} }

\date{May 2009}

\maketitle

\begin{abstract}
The competition graph of a doubly partial order is
known to be an interval graph.
The competition-common enemy graph
of a doubly partial order is also known to be an interval graph
unless it contains a cycle of length $4$ as an induced subgraph.
In this paper, we show that the niche graph of a doubly partial order
is not necessarily an interval graph.
In fact, we prove that, for each $n\ge 4$,
there exists a doubly partial order whose niche graph
contains an induced subgraph isomorphic to a cycle of length $n$.
We also show that if the niche graph of a doubly partial order
is triangle-free, then it is an interval graph.
\end{abstract}

\noindent
{\bf Keywords:}
niche graph;
doubly partial order;
interval graph

\newpage
\section{Introduction}

Throughout this paper, all graphs and all digraphs are simple.

Given a digraph $D$, if $(u,v)$ is an arc of $D$,
we call $v$ a {\em prey} of $u$ and $u$ a {\em predator} of $v$. The {\em competition graph} $C(D)$ of a digraph $D$
is the graph which
has the same vertex set as $D$
and has an edge between vertices $u$ and $v$
if and only if there exists a common prey of $u$ and $v$ in $D$.
The notion of competition graph is due to Cohen~\cite{cohen1} and has
arisen from ecology.
Competition graphs also have applications in coding, radio
transmission, and modelling of complex economic systems.
(See \cite{RayRob} and \cite{Bolyai} for a summary of these applications.)
Since Cohen introduced the notion of competition graph,
various variations have been defined
and studied by many authors (see the survey articles by
Kim~\cite{Kim93} and Lundgren~\cite{Lundgren89}).
One of its variants,
the {\em competition-common enemy graph} (or {\em CCE graph})
 of a digraph $D$ introduced by Scott~\cite{sc}
is the graph which
has the same vertex set as $D$ and
has an edge between vertices $u$ and $v$ if and only if
there exist both a common prey and a common predator of $u$ and $v$ in $D$.
Another variant, the {\em niche graph} of a digraph $D$ introduced by
Cable~{\it et al.}~\cite{cable} is the graph which
has the same vertex set as $D$ and
has an edge between vertices $u$ and $v$ if and only if
there exists a common prey or a common predator of $u$ and $v$ in $D$.

A graph $G$ is an {\em interval graph}
if we can assign to
each vertex $v$ of $G$ a real interval $J(v) \subset \mathbb{R}$
such that whenever $v \neq w$,
\[
vw \in E
\text{ if and only if }
J(v) \cap J(w) \neq \emptyset.
\]
The following theorem is a well-known characterization for interval graphs.

\begin{Thm}[\cite{lb}]
A graph is an interval graph if and only if
it is a chordal graph and it has no asteroidal triple.
\label{invervalChara}
\end{Thm}

Cohen~\cite{cohen1, cohen2} observed empirically that
most competition graphs of acyclic digraphs representing food webs
are interval graphs. Cohen's observation and the continued
preponderance of examples that are interval graphs led to a large
literature devoted to attempts to explain the observation and to
study the properties of competition graphs. Roberts~\cite{Rob78}
showed that every graph can be made into the competition graph of an
acyclic digraph by adding isolated vertices. (Add a vertex
$i_{\alpha}$ corresponding to each edge $\alpha=\{a,b\}$ of $G$,
and draw arcs from $a$ and $b$ to $i_{\alpha}$.) He then asked for
a characterization of acyclic digraphs whose competition graphs
are interval graphs.
The study of acyclic digraphs whose competition graphs are
interval graphs led to several new problems and applications (see
\cite{fi, Fraughnaugh,kimrob,lunmayras}).

We introduce some notations for simplicity. A cycle of length $n$
is denoted by $C_n$.
For two vertices $x$ and $y$ in a graph $G$,
we write $x \sim y$ in $G$ when $x$ and $y$ are adjacent in $G$.
For each point $x$ in $\mathbb{R}^2$, we denote its first
coordinate by $x_1$ and the second coordinate by $x_2$.

We define a partial order $\prec$ on $\mathbb{R}^2$ by
\[
x \prec y \text{ if and only if } x_1 < y_1 \mbox{ and } x_2 < y_2.
\]

For $x, y, z \in \mathbb{R}^2$,
$x,y \prec z$ (resp. $x,y \succ z$) means $x \prec
z$ and $y \prec z$ (resp.\ $x \succ z$ and $y \succ z$).
For vertices
$x$ and $y$ in $\mathbb{R}^2$,
we write
\[
\begin{array}{ll}
x \searrow y & \quad \text{ if } x_1 \le y_1 \text{ and } y_2 \le x_2 \\
x \preceq y & \quad \text{ if } x_1 \le y_1 \text{ and } x_2 \le y_2.
\end{array}
\]

A digraph $D$ is called a {\it doubly partial order}
if there exists a finite subset $V$ of $\mathbb{R}^2$ such that
\[
V(D) = V \text{ and } A(D) = \{(v,x) \mid v,x\in V, x \prec v \}.
\]

We may embed each of the competition graph,
the CCE graph, and the niche graph of a
doubly partial order $D$ in $\mathbb{R}^2$ by locating each vertex
at the same position as in $D$. We will always assume that $D$,
its competition graph,
CCE graph, and niche graph are embedded in
$\mathbb{R}^2$ in natural way.

\begin{figure}
\begin{center}
\psfrag{x}{$x$} \psfrag{y}{$y$}
  \includegraphics[width=200pt]{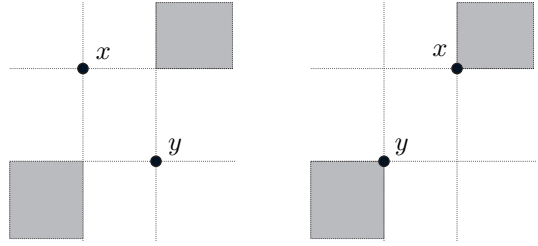}\\
  \caption{The region related to the adjacency of $x$ and $y$}
\label{EIR}
  \end{center}
\end{figure}

For two vertices $x$ and $y$ of a doubly partial order $D$,
if there is a vertex of $D$ in the region
\begin{eqnarray*}
&&
\{ z \in \mathbb{R}^2 \mid z \prec (\min\{x_1,y_1\}, \min\{x_2,y_2\}) \} \\
&\cup&
\{ z \in \mathbb{R}^2 \mid z \succ (\max\{x_1,y_1\}, \max\{x_2,y_2\}) \}
\end{eqnarray*}
(see Figure~\ref{EIR}),
then, by definition, $x$ and $y$ are adjacent in the niche graph of $D$.

The competition graph of a doubly partial order is
an interval graph, and
the CCE graph
of a doubly partial order is also an interval graph
if it is $C_4$-free:

\begin{Thm}[\cite{chokim}]
The competition graph of a doubly partial order is
an interval graph. \label{dpo}
\end{Thm}

\begin{Thm}[\cite{SJkim}]
The CCE graph of a doubly partial order is an interval graph
unless it contains $C_4$ as an induced subgraph.
\end{Thm}

It is natural to ask if another important variant of the competition graph,
the niche graph, of a doubly partial order is an interval graph.
In this paper, we show that for each $n\ge 4$,
there is a doubly partial order whose niche graph contains
an induced subgraph isomorphic to $C_n$,
which implies that the niche graph of a doubly partial order
is not necessarily an interval graph.
Then we show that if the niche graph of a doubly partial order
is triangle-free, then it is an interval graph.

\section{Main results}

We will show that the niche graph of a doubly partial order
is not necessarily an interval graph.
We first prove the following lemma.

For $c\in \mathbb{R}$, let $L_c :=\{ v \in \mathbb{R}^2  \mid  v_1+v_2=c \} $ and
$\mathbb{Z}^2:=\{v \in \mathbb{R}^2  \mid v_1,v_2\in \mathbb{Z} \}$.
Given a vertex $v$ in a graph $G$, we denote by $\Gamma_{G}(v)$ the neighborhood of $v$ in $G$.

\begin{Lem}\label{lemCn}
Let $V$ be a finite subset of $\mathbb{R}^2$
satisfying \[ V \cap \mathbb{Z}^2 \subseteq L_c \cup L_{c+2}
\quad \text{ and } \quad    V \setminus {\mathbb{Z}^2}  \subseteq \bigcup_{c< c'<c+2} L_{c'}\]
for some  $c\in \mathbb{R}$.
Suppose that $u_1 +1 \neq v_1$ or $u_2 -1 \neq v_2$ for two vertices $u$, $v$ of $V \cap \mathbb{Z}^2$ with $u_1 \leq v_1$.
Then  $u\not\sim v$ in the niche graph of
the doubly partial order $D$ associated with $V$.
\end{Lem}

\begin{proof}We prove by contradiction.
Suppose that there exist two vertices $u,v\in V \cap \mathbb{Z}^2 $
with $u_1 \le v_1$
such that $u_1+1\not=v_1$ or $u_2-1\not=v_2$ but $u \sim v$ in the niche graph of $D$.
Since $u \sim v$, there exists a vertex $a\in V$
such that either $a \prec u,v$ or $u,v \prec a$.
Since $a\in V$,
\begin{eqnarray} \label{eq1}
&&c \le a_1+a_2 \le c+2.
\end{eqnarray}

Suppose that $\{u,v\} \not\subset L_c$  and $\{u,v\} \not\subset L_{c+2}$.
Then either $u\in L_{c+2}$ and $v\in L_c$, or $u\in L_{c}$ and $v\in L_{c+2}$.
This implies that
\[\min\{u_1+u_2,v_1+v_2 \}=c\quad \text{ and }\quad \max\{u_1+u_2,v_1+v_2\}=c+2.\]
If $a\prec u,v$, then $a_1+a_2<\min\{u_1+u_2,v_1+v_2 \}=c $, which contradicts (\ref{eq1}).
If $u,v \prec a$, then $a_1+a_2> \max\{u_1+u_2,v_1+v_2\}=c+2$, which contradicts (\ref{eq1}) again.
Therefore  either $\{u,v\} \subset L_c$ or $\{u,v\} \subset L_{c+2}$.

Now suppose that $\{u,v\} \subset L_c$.
If $a \prec u,v$, then $a_1+a_2 <u_1+u_2=c$,
which is a contradiction to (\ref{eq1}).
Therefore it must hold that $u,v \prec a$.  Then it is easy to check that
\begin{equation}
a_1+a_2 > v_1+u_2.
\label{eqn}
\end{equation}
Since  $u \neq v$ and $c=u_1+u_2=v_1+v_2$, $u_1\not=v_1$.
By the assumption that $u_1\le v_1$, it is true that $u_1<v_1$.  Since $c=u_1+u_2=v_1+v_2$, $u_2>v_2$.
In addition, from the assumption that  $u_1 + 1 \not= v_1$ or $u_2-1\not=v_2$, we have $v_1-u_1\ge 2$ or $u_2-v_2 \ge 2$.
If $v_1-u_1\ge 2$, then, by (\ref{eqn}), $a_1+a_2>v_1+u_2\ge u_1+u_2+2=c+2$, which contradicts (\ref{eq1}).
If  $u_2-v_2 \ge 2$, then,  by (\ref{eqn}),  $a_1+a_2>v_1+u_2\ge v_1+v_2+2=c+2$, which is a contradiction.
Therefore it must hold that $\{u,v\} \subset L_{c+2}$.

If $u,v \prec a$, then $c+2=u_1+u_2<a_1+a_2$, which is a contradiction to (\ref{eq1}).
Therefore it must hold that $a \prec u,v$. Then
\begin{equation}
a_1+a_2 < u_1+v_2.
\label{eqns}
\end{equation}
Since  $u \neq v$, $u_1 \le v_1$, and $c+2=u_1+u_2=v_1+v_2$, it is true that $u_1<v_1$ and $v_2>u_2$.
Since $u_1 + 1 \not= v_1$ or $u_2-1\not=v_2$,
we have  $v_1-u_1\ge 2$ or $u_2-v_2 \ge 2$.
If $v_1-u_1\ge 2$, then, by (\ref{eqns}),  $a_1+a_2<u_1+v_2\le v_1+v_2-2=c$, which is a contradiction.
If $u_2-v_2 \ge 2$, then, by (\ref{eqns}), $a_1+a_2<u_1+v_2\le u_1+u_2-2=c$, which is a contradiction.

Hence $u$ and $v$ are not adjacent in the niche graph of $D$.
\end{proof}

\begin{Thm}\label{thm:Cn}
For any integer $n\ge 4$,
there is a doubly partial order whose niche graph
contains $C_n$ as an induced subgraph.
\end{Thm}


\begin{proof}
We construct a doubly partial order $D_n$ for each integer $n\ge 4$.
For any $(i,j)\in \mathbb{R}^2$,
let $X_{(i,j)}:=\{(i-1,j-1), (i,j), (i+1,j+1) \}$.
For an integer $k$ with $k\ge 2$, we define a finite
subset $W_k$ of $\mathbb{R}^2$ as follows:
\begin{align*}
W_k \cap \mathbb{Z}^2 &:= \{(i,k-1-i),(i+1,k-i)  \mid  i=0,1,\ldots, k-2  \} \\
W_k \setminus \mathbb{Z}^2 &:= \{ (i-\frac{1}{3},k-i-\frac{1}{3}), (i+\frac{1}{3},k-i+\frac{1}{3})
\mid  i=1,2,\ldots,k-2 \} \quad (k \ge 3)
\end{align*}
and $W_2 \setminus \mathbb{Z}^2=\emptyset$.
Let $A_k$ be the sequence
of vertices of $(W_k \cap \mathbb{Z}^2) \cup\{(0,k)\}$ listed as  follows:
\begin{align*}
&(k-2,1),(k-3,2),\ldots,(i,k-1-i),\ldots,(2,k-3),(1,k-2),(0,k-1), \tag{$*$}\\
&(0,k), (1,k),(2,k-1),\ldots,(i+1,k-i), \ldots,(k-2,3),(k-1,2).
\end{align*}

Let $G_k$ be the niche graph of a doubly partial order
associated with $X_{(0,k)} \cup W_k$.
First, we will show that the sequence $A_k$ is a path of
length $2k-2$ as an induced subgraph in $G_k$.
In $G_k$, we can easily check the following:
\begin{itemize}
\item[(i)]
For $i=0,1,\ldots,k-3$,
the vertex $(i+1+\frac{1}{3},k-1-i+\frac{1}{3})$  of $W_k \setminus \mathbb{Z}^2$
is a common predator of the $(k-1-i)$th vertex $(i,k-1-i)$ and the $(k-i)$th vertex $(i+1,k-2-i)$;
\item[(ii)]
For $i=0,1,\ldots,k-3$, the vertex
$(i+1-\frac{1}{3},k-1-i-\frac{1}{3})$ of $W_k \setminus \mathbb{Z}^2$ is a common prey
of the $(k+i+1)$st vertex $(i+1,k-i)$ and the $(k+i+2)$nd vertex $(i+2,k-1-i)$;
\item[(iii)]
The vertex $(1,k+1)$ is a common predator
of the $k$th vertex $(0,k)$ and the $(k-1)$st vertex $(0,k-1)$;
\item[(iv)]
The vertex $(-1,k-1)$ is a common prey of the $k$th vertex $(0,k)$ and the $(k+1)$st vertex $(1,k)$.
\end{itemize}
By (i) through (iv), the $i$th vertex and the $j$th vertex of  the sequence $A_k$ are adjacent in $G_k$ if $|i-j|=1$, and so $A_k$
forms a path of length $2k-2$ in $G_k$.

In addition, the sequence $A_k$ is a path of length $2k-2$  as an induced subgraph in $G_k$.
To see why, we will show that  the $i$th vertex and the $j$th vertex of $A_k$ are not adjacent in $G_k$ if  $|i-j|\ge 2$.
Take the $i$th vertex and the $j$th vertex of $A_k$ with  $|i-j|\ge 2$ and denote them by $x$ and $y$.
Suppose that $k=i$ or $j$. Then the $k$th vertex of $A_k$ is $(0,k)$ and it is easy to check that
\[\Gamma_{G_k}((0,k))=\{(1,k), (0,k-1), (-1,k-1), (1,k+1)\}.\] Since $(1,k)$ and $(0,k-1)$ are the $(k+1)$st vertex and $(k-1)$st vertex of $A_k$, respectively, and $(-1,k-1)$ and $(1,k+1)$ are not vertices of $A_k$, we conclude that $x \not\sim y$ in this case.

Suppose that $i \neq k$ and $j \neq k$. Without loss of generality, we may assume that $x_1 \le y_1$.
Note that $W_k$ satisfies that
\[W_k\cap \mathbb{Z}^2 \subseteq L_{k-1} \cup L_{k+1}
\quad \text{ and } \quad
W_k \setminus \mathbb{Z}^2 \subseteq \bigcup_{k-1< c'<k+1} L_{c'}.\]
Since $|i-j|\ge 2$,  $x_1 + 1 \not= y_1$ or $x_2-1\not=y_2$ by the definition of $A_k$.  Then, by Lemma~\ref{lemCn}, $x\not\sim y$ in the niche graph of the doubly partial order associated with $W_k$.
Therefore  $x\not\sim y$ in the subgraph of $G_k$ induced by  $W_k$.
It remains to show that $x$ and $y$ have neither a common prey nor a common predator in $X_{(0,k)}=\{(-1,k-1),(0,k),(1,k+1)\}$. The set of predators or prey of $(-1,k-1)$ in $A_k$ is $\{(0,k),(1,k)\}$.  These two vertices are $k$th and $(k+1)$st vertices of $A_k$ and so $(-1,k-1)$ cannot be a common prey or a common predator of $x$ and $y$.  The set of predators or prey of $(0,k)$ in $A_k$ is $\{(-1,k-1), (1,k+1)\}$ and so $(0,k)$ cannot be a common prey or a common predator of $x$ and $y$.  The set of predators or prey of $(1,k+1)$ in $A_k$ is $\{(0,k),(0,k-1)\}$.  These two vertices are $k$th and $(k-1)$st vertices of $A_k$ and so $(-1,k-1)$ cannot be a common prey or a common predator of $x$ and $y$.
Hence we conclude that the $i$th vertex and the $j$th vertex of $A_k$ are not adjacent in $G_k$ if  $|i-j|\ge 2$.
\begin{figure}
\begin{center}
\psfrag{d}{$(0,4)$}
\psfrag{x}{$(1,5)$}
\psfrag{y}{$(-1,3)$}
\psfrag{e}{$(1,4)$}
\psfrag{c}{$(0,3)$}
\psfrag{b}{$(1,2)$}
\psfrag{f}{$(2,3)$}
\psfrag{a}{$(2,1)$}
\psfrag{g}{$(3,2)$}
\psfrag{r}{$(2,0)$}
\psfrag{s}{$(3,1)$}
\psfrag{t}{$(4,2)$}
 { \includegraphics[width=230pt]{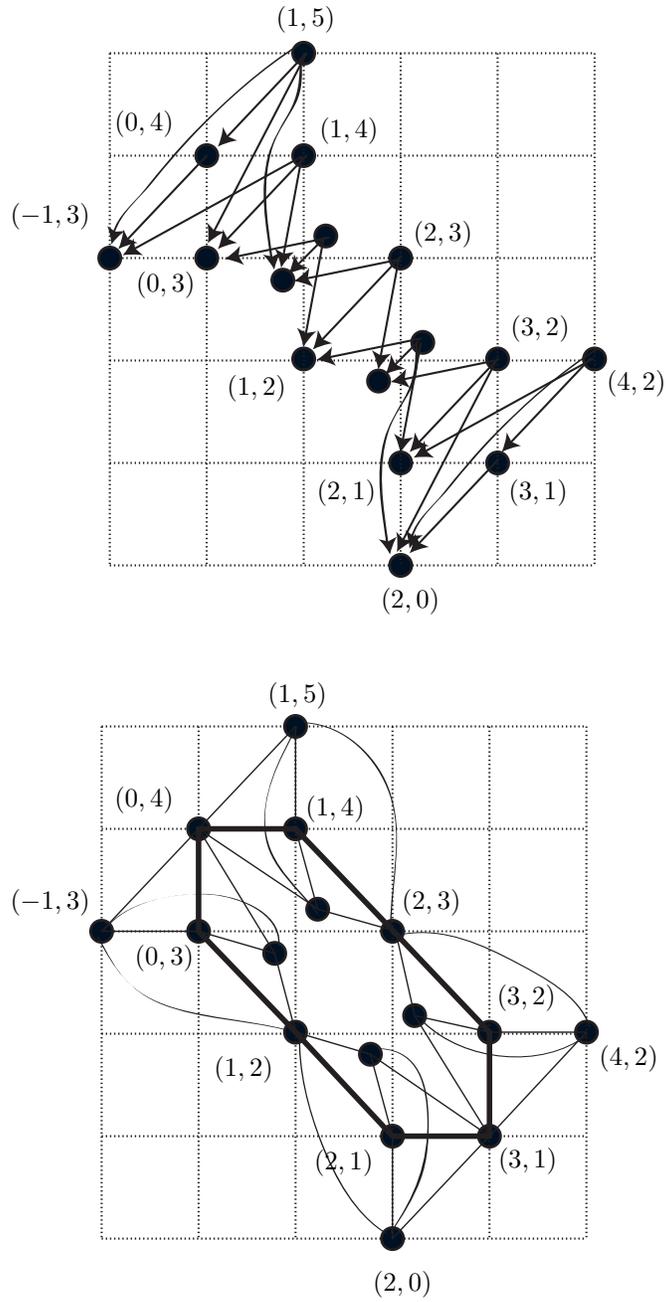} }\\
  \caption{A doubly partial order $D_8$ and the niche graph of $D_8$. Note that the thick edges form a cycle of length $8$ as an induced subgraph of the graph. }
\label{C8}
  \end{center}
\end{figure}

\begin{figure}
\begin{center}
\psfrag{d}{$(0,4)$}
\psfrag{x}{$(1,5)$}
\psfrag{y}{$(-1,3)$}
\psfrag{e}{$(1,4)$}
\psfrag{c}{$(0,3)$}
\psfrag{b}{$(1,2)$}
\psfrag{f}{$(2,3)$}
\psfrag{a}{$(2,1)$}
\psfrag{g}{$(3,2)$}
\psfrag{r}{$(4,0)$}
\psfrag{s}{$(5,1)$}
\psfrag{t}{$(6,2)$}
{\includegraphics[width=270pt]{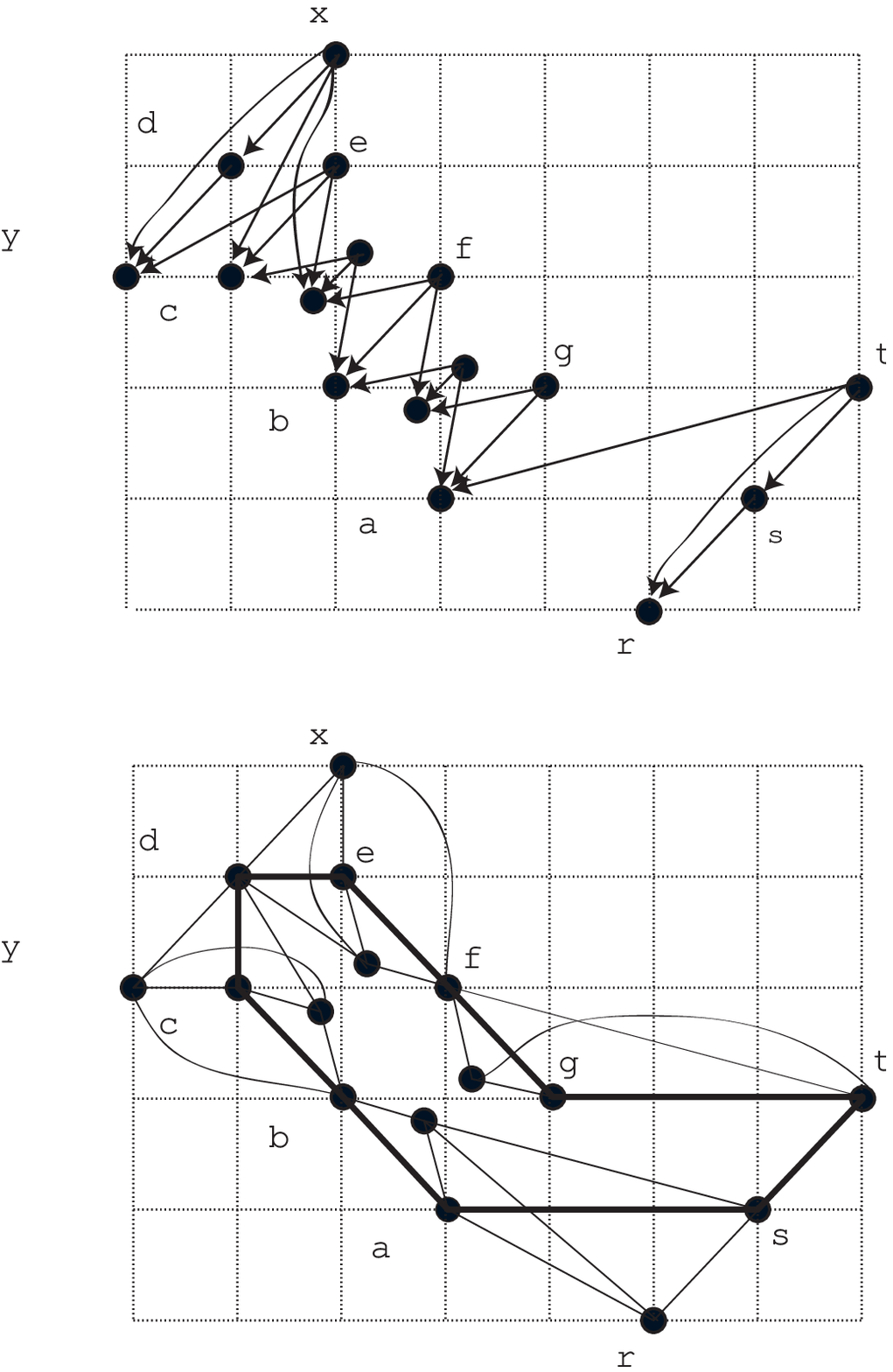} }\\
  \caption{A doubly partial order $D_9$ and the niche graph of $D_9$. The thick edges form a cycle of length $9$ as an induced subgraph of the graph.}
\label{C9}
  \end{center}
\end{figure}

Now we are ready to give a construction of  a doubly partial order $D_n$ for each integer $n\ge 4$.
Suppose that $n=2k$ for some integer $k\ge 2$.
Let \begin{align*}
V_n:&= X_{(0,k)} \cup X_{(k-1,1)} \cup W_k
\end{align*}
and $D_n$ be the doubly partial order associated with $V_n$.
We will show that the vertices of $(W_k \cap \mathbb{Z}^2) \cup \{ (0,k), (k-1,1) \}$
form $C_n$ without chord in the niche graph of $D_n$.
See Figure~\ref{C8} for an illustration. Let $N_n$ be the niche graph of $D_n$.

Note that  $X_{(k-1,1)}=\{ (k-2,0),(k-1,1),(k,2) \}$. Consider the  sequence $A_{k}$ defined in ($*$). It is not difficult to check that none of vertices in $X_{(k-1,1)}$ can be a common prey or a common predator of two vertices of $A_k$.  Thus by the previous argument, $A_k$ forms a path as an induced subgraph of $N_n$.  On the other hand,
in the niche graph $N_n$ of $D_n$, it can easily be checked that
\begin{eqnarray*}
 && \Gamma_{N_n}( (k-2,0)) =\{ (k-1,1),(k,2),(k-1,2) \};\\
 && \Gamma_{N_n}( (k,2) ) =\{ (k-1,1),(k-2,0),(k-2,1) \};\\
 && \Gamma_{N_n}( (k-1,1) )=\{ (k-2,0),(k,2),(k-2,1),(k-1,2)\}.
\end{eqnarray*}
Thus, the vertices of $A_k$ together with $(k-1,1)$ form a cycle of
length $2k=n$ as an induced subgraph.

Now we assume that $n$ is an odd integer with $n\ge 5$.
Then  $n=2k+1$ for some integer $k\ge 2$. Let
\begin{align*}
V_n:&= X_{(0,k)} \cup X_{(k+1,1)}\cup W_k
\end{align*}
and $D_n$ be the doubly partial order associated with $V_n$.  See Figure~\ref{C9} for an illustration. Note that  $X_{(k+1,1)}=\{ (k,0),(k+1,1),(k+2,2) \}$.

Consider the sequence $A_k$ defined in $(\ast)$.  Then it is not hard to check that none of vertices in $X_{(k+1,1)}$ is a common prey or a common predator of two vertices of $A_k$. Thus, by the previous argument, $A_k$ is a path as an induced subgraph of $N_n$.

It can easily be checked that
\begin{eqnarray*}
 && \Gamma_{N_n}( (k,0)) =\{ (k+1,1),(k-2,1) \};\\
 && \Gamma_{N_n}( (k+1,1) ) =\{(k,0),(k+2,2),(k-2,1) \};\\
 && \Gamma_{N_n}( (k+2,2) )=\{ (k+1,1),(k-1,2)\}.
\end{eqnarray*}
Thus the first vertex $(k-2,1)$ of $A_k$ is the only vertex in $A_k$ adjacent to $(k+1,1)$.  In addition, the $(2k-1)$st vertex $(k-1,2)$ of $A_k$ are the only vertex in $A_k$ adjacent to $(k+2,2)$. Since $(k+1,1)$ and $(k+2,2)$ are adjacent, the vertices of  sequence $A_k$ together with $(k+2,2)$ and $(k+1,1)$ form a cycle of length $2k+1=n$ as an induced subgraph.
Hence $N_n$ contains $C_n$ as an induced subgraph.
\end{proof}

Theorems~\ref{invervalChara} and ~\ref{thm:Cn} tell us
that the niche graph of a doubly partial order
is not necessarily an interval graph.
However if the niche graph of a doubly partial order
is triangle-free, then it is an interval graph.
To show that, we start with the following lemma:

\begin{Lem} \label{lem1}
Let $D$ be a doubly partial order.
Suppose that  the niche graph $G$ of $D$ is triangle-free.
Then if $x \sim y$, $y \sim z$ in $G$, and $x_1 \le z_1$,
then  $x \searrow y \searrow z$.
\end{Lem}

\begin{proof}
Since  $x \sim y$ and $y \sim z$ in $G$,
there are vertices $a$ and $b$
such that either $a \prec x, y$ or $x ,y\prec a$
and either $b\prec y,z$ or $y,z \prec b$.
Suppose that $a\prec x, y$ and $y,z\prec b$.
Then $a \prec y \prec b$ and so $a\prec b$.
Therefore  $a$ is a common prey of $x$, $y$, and $b$,
and so $x$, $y$ and $b$ form a triangle in $G$,
which is a contradiction.
Similarly, if $ x, y \prec a$ and $b\prec y, z$,
then we reach a contradiction.
Hence either
(1) $a\prec x, y$ and  $b\prec y, z$,
or (2) $x,y \prec a$ and $y,z\prec b$.
In each case,
we show that $x_1 \le y_1 \le z_1$. To show by contradiction, we consider two subcases (A) $x_1>y_1$ and (B) $y_1>z_1$ in each case.

\noindent
{\it Case 1.}
$a\prec x, y$ and $b\prec y, z$.


\noindent
{\it Subcase A.}
$y_1 < x_1$.

If $z_2 \le x_2$,
then $b_1< y_1 < x_1$ and $b_2< z_2 \le x_2$
which imply that $b\prec x$.
Then $b\prec x,y,z$
and so $x$, $y$, and $z$ form a triangle in $G$, which is a contradiction.
If $z_2 > x_2$,
then $a_1< y_1 \le x_1 \le z_1$ and $a_2< x_2 < z_2$
which imply that $a\prec z$.
Then $a \prec x,y,z$
and so $x$, $y$, and $z$ form a triangle in $G$, which is a contradiction.

\noindent
{\it Subcase B.} $z_1< y_1$.

If $x_2<y_2$,
then $x \prec y$
and so $x,a,b\prec y$.
Now suppose that $ y_2 \le x_2$ and $ y_2\le z_2$.
If $x_1 \le z_1$,
then $a_1<  x_1  \le  z_1 $ and $a_2< y_2 \le z_2 $,
which imply that $a\prec z$.
Then $a\prec x,y,z$
and so $x$, $y$, and $z$ form a triangle in $G$,
which is a contradiction.
If $z_2<y_2$,
then $z \prec y$
and so $z,a,b \prec y$.
Now suppose that $ y_2 \le x_2$ and $ y_2\le z_2$.
If $z_1 < x_1$,
then $b_1< z_1 < x_1 $ and $b_2< y_2 \le x_2 $,
which imply that $b\prec x$.
Then $b \prec x,y,z$
and so $x$, $y$, and $z$ form a triangle in $G$,
which is a contradiction.

\noindent
{\it Case 2.}
$x,y \prec a$ and $y,z\prec b$.


\noindent
{\it Subcase A.} $y_1<x_1$.

If $y_2 < x_2$,
then $y \prec x$
and so $y \prec x,a,b$.
Then $x$, $a$, and $b$
form a triangle, which is a contradiction.
If $y_2 < z_2$,
then $y \prec z$
and so $y \prec z,a,b$.
Then $z$, $a$, and $b$
form a triangle, which is a contradiction.
Now suppose that $x_2 \le y_2$ and $z_2 \le y_2$.
If $x_1 \le z_1$,
then $x_1  \le  z_1 < b_1$ and $x_2 \le y_2 < b_2$,
which imply that $x\prec b$.
Then $x,y,z\prec b$
and so $x,y$ and $z$ form a triangle in $G$,
which is a contradiction.
If $z_1 < x_1$,
then $z_1 < x_1 < a_1$ and $z_2 \le y_2 < a_2$,
which imply that $z\prec a$.
Then $x,y,z\prec a$
and so $x,y$ and $z$ form a triangle in $G$,
which is a contradiction.

\noindent
{\it Subcase B.} $z_1<y_1$.

If $x_2 < z_2$,
then $x_1 \le y_1 < b_1$ and $x_2< z_2 \le b_2$
which imply that $x\prec b$.
Then $ x,y,z\prec b$
and so $x$, $y$, and $z$ form  a triangle in $G$,
which is a contradiction.
If $x_2 \ge z_2$,
then $z_1 \le y_1 <a_1$ and $ z_2 \le x_2  <a_2$
which imply that $z\prec a$.
Then $ x,y,z\prec a$
and so $x$, $y$, and $z$ form  a triangle in $G$,
which is a contradiction.
\smallskip

Thus we can conclude that $x_1\le y_1 \le z_1$ in each case.
In addition, it cannot happen $x_1=y_1=z_1$.
To see why,
let $c$ be an element of $\{a,b\}$ with smallest second component
and $d$ be the element of $\{a,b\} \setminus \{c\}$.
Suppose that $a\prec x,y$ and $b\prec y,z$.
Since $x_1=y_1=z_1$, we have $c \prec x,y,z$ and so
$x$, $y$, and $z$ form a triangle.
Similarly, if $ x,y\prec a$ and $y,z \prec b$,
then $x,y,z \prec d$
and so $x,y,z$ create a triangle.
Therefore
it holds that (1) $x_1=y_1<z_1$, (2) $x_1<y_1=z_1$, or (3) $x_1<y_1<z_1$.
In the following, we show that $x_2 \ge y_2 \ge z_2$
in these three cases.
\smallskip

\noindent
\textit{Case 1.}
$x_1=y_1<z_1$

Suppose that $x_2<y_2$.
If $ x,y\prec a$ and $y,z \prec b$,
then $x,y,z \prec b$.
If $a\prec x,y$ and $b\prec y,z$,
and $z_2<x_2$, then $b\prec x,y,z$.
If $a\prec x,y$ and $b\prec y,z$,
and $z_2\ge x_2$, then $a\prec x,y,z$.
Therefore we reach a contradiction, and so it must hold that $x_2\ge y_2$.
Suppose that $y_2 < z_2$.
If $a\prec x,y$ and $b\prec y,z$,
then, since $b_1<y_1=x_1$ and $b_2<y_2\le x_2$,
we have $b\prec x,y,z$.
If $x,y\prec a$ and $y,z \prec b$,
then, since $y \prec a,b$ and $y \prec z$, we have $y\prec a,b,z$.
Therefore we reach a contradiction,
and so it must hold that $y_2\ge z_2$.
Thus $x_2\ge y_2 \ge z_2$.
\smallskip

\noindent
\textit{Case 2.}
$x_1<y_1=z_1$.

Suppose that $y_2<z_2$.
If $a\prec x,y$ and $b\prec y,z$, then $a\prec x,y,z$.
If $ x,y\prec a$ and $y,z \prec b$ and $z_2\ge x_2$,
then $ x,y,z\prec b$.
If $ x,y\prec a$ and $y,z \prec b$ and $z_2< x_2$, then $x,y,z\prec a$.
Therefore we reach a contradiction, and so it must hold that
$y_2\ge z_2$.
Suppose that $x_2 <y_2$.
If $ x,y\prec a$ and $y,z \prec b$,
then, since $z_1=y_1<a_1$ and $ z_2 \le y_2<a_2$,
we have $x,y,z\prec a$.
If $a \prec x,y$ and $b \prec y,z$,
then, since $a,b \prec y$ and $x \prec y$,
we have $x,a,b\prec y$.
Therefore we reach a contradiction, and so
it must hold that $x_2\ge y_2$.
Thus $x_2\ge y_2 \ge z_2$.

\noindent
\textit{Case 3.}
$x_1<y_1<z_1$.

Suppose that $x_2 < y_2$.
Then $x \prec y$.
If $a\prec x, y$ and  $b\prec y, z$,
then $a,x,b \prec y$.
If $x,y \prec a$ and $y,z\prec b$, then $x,y,z \prec b$.
Therefore we reach a contradiction,
and so $x_2 \ge y_2$.
Suppose that $y_2 < z_2$.
Then $y \prec z$.
If $a\prec x, y$ and  $b\prec y, z$,
then $a, b, y \prec z$.
If $x,y \prec a$ and $y,z\prec b$, then $y \prec a, b, z$.
Therefore we reach a contradiction,
and so $y_2 \ge z_2$.
Thus $x_2\ge y_2 \ge z_2$.

Hence we conclude that $x_2 \ge y_2 \ge z_2$
and so $x \searrow y \searrow z$.
\end{proof}
\begin{Thm}\label{thm:forest1}
Let $D$ be a doubly partial order.
Suppose that the niche graph of $D$ is a triangle-free graph.
Then each component of the niche graph of $D$ is a path.
\end{Thm}

\begin{proof}
Let $G$ be the niche graph of a doubly partial order $D$.
First, we will show that  $G$ is a forest.
Suppose that there is a cycle $C$ of length $n$.
 We may assume that $x$ is a vertex  such that its  first component $x_1$
is the minimum among those of vertices of $C$.
Since $G$ is triangle-free, $n\ge 4$ and so
there exist  $4$ distinct  vertices $x,y,z,w$ such that $x\sim y$,  $y \sim z$, $w \sim x$.
Let $u$ be the vertex of $C$ such that $u\sim w$ and $u\not=x$. 
By the choice of $x$, $x_1 \le u_{1}$ and $x_1 \le z_{1}$.
Then, since $xwu$ and $xyz$ are paths in $G$, $x\searrow w$ and $x \searrow y$ by Lemma \ref{lem1}.
If $y_1 \ge  w_1$, then, by Lemma \ref{lem1}, $w  \searrow x$,
which implies that $x=w$.
If $y_1 < w_1$, then $y \searrow x$,
which implies that $y=x$.
Thus we reach a contradiction in either case. Hence $G$ is a forest.

In the following, we will show that  $\deg_{G}(v) \le 2$ for any vertex $v$.
Suppose that there is a vertex $u$ such that $\deg_{G}(u) \ge 3$.
Let $x$, $y$ and $z$ be three distinct neighbors of $u$.
Without loss of generality, we may assume that $x_1\le y_1 \le z_1$.
Since $xuy$ and $yuz$ are paths in $G$,
$x \searrow u \searrow y$ and $y \searrow u \searrow z$
by Lemma \ref{lem1}.
Then $u \searrow y$ and $y \searrow u$ and so $y=u$,
which is a contradiction.
Hence each component of the niche graph of $D$ is a path.
\end{proof}

By Theorem~\ref{invervalChara} and Theorem~\ref{thm:forest1},
the following theorem holds.

\begin{Thm} \label{thm:interval}
The niche graph of a doubly partial order is an interval graph
unless it contains a triangle.
\end{Thm}

\section{Concluding remarks}

We have shown that the niche graph of a doubly partial order
is not necessarily an interval graph by constructing a
doubly partial order whose niche graph contains a cycle
an induced subgraph for each integer $n \ge 4$.  Then we tried to find a doubly partial order such that its niche graph does not contain a cycle of length at least $4$ as an induced subgraph and it is not an interval graph, but  in vain. Accordingly, we would like to ask whether or not such a doubly partial order exists.

Eventually, it remains open to characterize doubly partial orders whose niche graphs are interval graphs.



\end{document}